\documentclass[12pt,reqno]{article}
\usepackage{hyperref}
\usepackage{amsmath, amsthm, graphicx,amsfonts, amssymb,color}
\usepackage{bm}

\newtheorem{theorem}{Theorem}[section]
\newtheorem{corollary}[theorem]{Corollary}
\newtheorem{lemma}[theorem]{Lemma}
\newtheorem{proposition}[theorem]{Proposition}
\theoremstyle{definition}

\theoremstyle{remark}
\newtheorem{remark}[theorem]{Remark}

\numberwithin{equation}{section}
\usepackage{mathrsfs}
\newcommand{\scr}[1]{\mathscr #1}

\def\cM{\mathcal M}

\def\bg{\begin}

\def\de{\end{equation}}

\def\edar{\end{eqnarray}}

\def\lb{\label}

\def\l{\left}\def\r{\right}

\def\q{\quad}

\def\lan{\langle}\def\ran{\rangle}

\def\[{\l[} \def\]{\r]}
\def\({\l(} \def\){\r)}

 \def\beqlb{\begin{eqnarray}}\def\eeqlb{\end{eqnarray}}
 \def\beqnn{\begin{eqnarray*}}\def\eeqnn{\end{eqnarray*}}
\def\d{{\mbox{\rm d}}}

\title{\bf  {Dirichlet eigenvalues and exit time moments for symmetric Markov processes}}

\author{
{\bf Lu-Jing Huang}\\
{\small School of Mathematics and Statistics, Fujian Normal University }\\
{\bf Tao Wang}\footnote{Corresponding author: Tao Wang.}\\
{\small  Laboratory of Mathematics and Complex Systems (Ministry of Education),}\\ {\small School of Mathematical Sciences, Beijing Normal University}
}

\date{}

\begin{document}

 \maketitle


\bg{abstract}

We give some relationships between the first Dirichlet eigenvalues and the exit time moments for the general symmetric Markov processes.
As applications, we present some examples, including symmetric diffusions and $\alpha$-stable processes, and provide the estimates of their first Dirichlet eigenvalues and the exit time moments.

\end{abstract}



\section{Introduction and main results}
The relationship between the first Dirichlet eigenvalues and the exit time moments of Markov processes is an important topic of recent interest in probability theory and geometry, for instance,  for a Brownian motion
on  Riemannian manifold, the Dirichlet eigenvalue which is an important geometric quantity, and the exit time moments,  known as torsional rigidity (see \cite{DLM17}) or moment spectrum (see \cite{McD13}), have a numerous  studies in the literature. More specifically,  \cite{Polya48,PS51} obtain the classical P\'{o}lya's inequality for the Dirichlet eigenvalue and the  torsional rigidity, further results can be found in e.g. \cite{BBV15,BFNT16}. McDonald \cite{McD02,McD13} used the mean exit time moment spectrum to analysis the Dirichlet spectrum, while \cite{DLM17,HMP16} prove a number of inequalities of them.

The purpose of this paper is to extend the problem to general symmetric Markov processes, and develop some new relationships between their first Dirichlet eigenvalues and exit time moments.


To state our main results,  we start with some notations. Let $(E,d,\mu)$ be a metric measure space, that is, $(E, d)$ is a Polish space with Borel $\sigma$-algebra $\mathscr{B}(E)$, and $\mu$ is a positive Radon measure on $E$ with full support.
Let $X:=(X_t)_{t\geqslant0}$  be a Markov process on $(E,d)$ with transition kernel $P_t(x,\d y),t\geqslant 0,x,y\in E$. We assume that $X$ is symmetric  with respect to $\mu$, that is,
$$
\mu(\d x)P_t(x,\d y)=\mu(\d y)P_t(y,\d x)\q \text{for all }t\geqslant0,\ x,y\in E.
$$
Denote by $(\cal{L},\scr{D}(\cal{L}))$ the generator of process $X$ in $L^2(\mu)$,  and by $(\cal{E},\cal{F})$ the associated symmetric Dirichlet form. It is known that
$$
\mathcal{E}(f,g)=-\langle \mathcal{L}f,g\rangle_\mu,\quad f,g\in \scr{D}(\cal L).
$$
Assume additionally that $(\cal{E},\cal{F})$ is regular, i.e., $\mathcal{F}\cap C_0(E)$ is dense both in $\cal{F}$ and $C_0(E)$,  where $C_0(E)$ is  the set of continuous functions with compact supports on $E$.


Let  $\Omega \subset E$ be a domain, $\mu(f):=\int_Ef\d \mu$. Define the first Dirichlet eigenvalue of process $X$ on $\Omega$ as
\begin{equation}\label{Diri-eigen}
\lambda_0(\Omega)=\inf \left\{\frac{\mathcal{E}(f,f)}{\mu(f^2)}: f \in \mathcal{F}  \text { and }\widetilde{f}=0\ \text{q.e. on }\Omega^{c}\right\},
\end{equation}
where $\widetilde{f}$ is the quasi-continuous version of $f$ and q.e. means quasi-everywhere, see \cite[Definition 1.2.12]{CF12} for more details.
From the regularity of $(\mathcal{E},\mathcal{F})$, it is clear that
\begin{align*}
\lambda_0(\Omega)&=\inf \left\{\frac{\mathcal{E}(f,f)}{\mu(f^2)}: f \in \mathcal{F}\cap C_0(\Omega)\right\}\\
&=\inf \left\{\frac{\mathcal{E}(f,f)}{\mu(f^2)}: f \in \mathcal{F}  \text { and }\widetilde{f}=0\ \mu\text{-a.e. on }\Omega^{c}\right\}.
\end{align*}
Denote by $\tau_\Omega=\inf\{t\geqslant 0: X_t\notin \Omega\}$ the first exit time of process $X$ from $\Omega$,
and by $T_{k}(\Omega)$ the $k$-th moment of $\tau_\Omega$, i.e.,
$$
T_{k}(\Omega)=\mathbb{E}_\mu[\tau_\Omega^k]:=\int_{\Omega} \mathbb{E}_{x}[\tau_\Omega^{k}] \mu (\d x).
$$


We can now state our first main result in the following, which gives some upper bounds of the first Dirichlet eigenvalue $\lambda_0(\Omega)$.
\begin{theorem}\label{main}
	Let $\Omega\subset E$ be a domain with $\mu(\Omega)<\infty$ and $\lambda_0(\Omega)>0$. Then for every $k\geqslant 1$,
	\begin{equation}\label{moment1}
		\lambda_{0}(\Omega) \leqslant \frac{(k !)^{2}}{(2 k-1) !} \frac{T_{2 k-1}(\Omega)}{T_{k}^{2}(\Omega)}\mu(\Omega),
	\end{equation}
	and
	\begin{equation}\label{moment3}
		\lambda_{0}(\Omega) \leqslant  k \frac{T_{k-1}(\Omega)}{T_{ k}(\Omega)}.
	\end{equation}
Thus we have
 	\begin{equation}\label{k-moment}
T_{ k}(\Omega)\leqslant\frac{k!\mu(\Omega)}{\lambda_0(\Omega)^k},\quad \text{and}\quad \mathbb{E}_\mu[\mathrm{e}^{\beta\tau_{\Omega}}]\leqslant \left(1+\frac{\beta}{\lambda_0(\Omega)-\beta}\right)\mu(\Omega)
\end{equation}
for all $\beta\in (0, \lambda_0(\Omega))$.
\end{theorem}

 \begin{remark}
Note that \cite[Propositon 2]{HMP16} obtains the similar bounds as Theorem \ref{main} for Brownian motion when $\Omega$ is a geodesic ball. Furthermore, \cite{DLM17} shows the similar result for diffusions on general bounded domain $\Omega$.
\end{remark}

To study the lower bound of $\lambda_0(\Omega)$, we need some other notations. Let $(P_t^\Omega)_{t\geq 0}$ be the transition semigroup of process $X$ killed upon leaving $\Omega$,  that is,
$$
P_t^\Omega(x,A)=\mathbb{P}_x(X_t\in A, t<\tau_\Omega), \quad x\in E,\ A\in \mathscr{B}(E).
$$
Denote by $(\mathcal{L}^\Omega,\mathscr{D}(\mathcal{L}^\Omega))$ the associated generator in $L^2(\mu)$. In fact, it is well known that $\lambda_0(\Omega)=-\sup\{\lambda:\ \lambda\in\sigma(\mathcal{L}^\Omega)\}$ where $\sigma(\mathcal{L}^\Omega)$ is the spectrum of $\mathcal{L}^\Omega$ in $L^2(\mu)$.

We have the following result while  $\mathcal{L}^\Omega$ possesses discrete spectrum (see \cite[Theorem 1.2]{GW02} for some criteria of discrete spectrum).
\begin{theorem}\label{main-3}
	Let $\Omega\subset E$ be a domain with  $\mu(\Omega)<\infty$ and $\lambda_0(\Omega)>0$. Assume that  $-\mathcal{L}^\Omega$ has discrete spectrum in $L^2(\mu)$, denoted by
$$
0< \lambda_{0}(\Omega)\leqslant\lambda_{1}(\Omega)\leqslant\lambda_{2}(\Omega)\leqslant \cdots.
$$
Let $\varphi_i,\ i\geqslant 0$ be the associated eigenfunctions. Then for any $k\geqslant 1$,
	$$
	T_k(\Omega)=k!\sum_{i=0}^\infty \frac{\mu(\varphi_i)^2}{(\lambda_i(\Omega))^k}
\geqslant k!\frac{\mu(\varphi_0)^2}{(\lambda_0(\Omega))^k}.
	$$
Thus,
\begin{equation}\label{lowerb}
\lambda_0(\Omega)\geqslant \left(\frac{k!\mu(\varphi_0)^2}{T_k(\Omega)}\right)^{1/k}.
\end{equation}
\end{theorem}

 \begin{remark}
	
(1) We mention that when $X$ is a Brownian motion, the similar result as Theorem \ref{main-3} is presented in \cite[Theorem 1.3]{DLM17}.

(2) From \eqref{lowerb}, we could obtain a lower bound of the exponential moment of the exit time:

	$$	\mathbb{E}_\mu[\mathrm{e}^{\beta\tau_{\Omega}}]\geqslant\left(1+\frac{\beta}{\lambda_0(\Omega)-\beta}\right)\mu(\varphi_0)^2,
	$$
which is a extension of \cite[Corollary 1.6]{HKMW21+}.
\end{remark}

As an application of Theorems 1.1 and 1.3, we obtain the following result.

\begin{corollary}\label{limit}
	Let $\Omega\subset E$ be a domain with  $\mu(\Omega)<\infty$ and $\lambda_0(\Omega)>0$ Assume that  $-\mathcal{L}^\Omega$ has discrete spectrum in $L^2(\mu)$.  Then we have
	\begin{equation}\label{sup}
	\begin{split}
		\lambda_0(\Omega)&=\sup \left\{\lambda\geqslant 0: \limsup _{n \rightarrow \infty}\lambda^n \frac{T_n(\Omega)}{n!}<\infty\right\}.\\
        &=\inf \left\{\lambda \geqslant 0:  \liminf _{n \rightarrow \infty}\lambda^n \frac{T_n(\Omega)}{n!}>0  \right\}.
    \end{split}
\end{equation}
\end{corollary}

\begin{remark}
Note that when the process $X$ is a Brownian motion, \eqref{sup} is obtained by \cite[Theorem 1]{HMP16}.
\end{remark}

	
	
		



The remainder of the paper is structured as follows. In the next section, we present the relation of exit time and Poisson equation. From it we obtain variational formulas for the exit times. Section \ref{proofs} is devoted to proofs of the main results. In section \ref{examples}, we provide some examples.

\section{Exit time, Poisson equation and variational formula}\label{auxiliary}
In this section, we will do some preparations, which are helpful in the proof of our main results.

Recall that $X:=(X_t)_{t\geq 0}$ is a symmetric Markov process on Polish space $(E,d,\mu)$, with the generator $(\mathcal{L},\scr{D}(\mathcal{L}))$ in $L^2(\mu)$  and regular Dirichlet form  $(\mathcal{E},\mathcal{F})$. For a domain $\Omega\subset E$, define
$$
\mathcal{F}^\Omega=\left\{f\in\mathcal{F}:\ \widetilde{f}=0\ \text{q.e. on }\Omega^c\right\}.
$$

Now fix a domain $\Omega\subset E$ and a function $\xi\in \mathcal{F}^\Omega$.
Consider the following Poisson equation corresponding to $\cal L$:
\begin{equation}\label{poi}
\begin{cases}
	-\mathcal{L}u=\xi ,&\q \text{in }\Omega;\\
	u=0,&\q \text{q.e. on }\Omega^c.
\end{cases}
\end{equation}
We call function $u\in\mathcal{F}^\Omega$ is a weak solution to \eqref{poi} if
$$
\mathcal{E}(u, f)=\lan \xi,f\ran_2:=\int_\Omega\xi f\d\mu\q \text{for all }f\in\mathcal{F}^\Omega.
$$
In fact, if $\mu(\Omega)<\infty$ and $\lambda_0(\Omega) >0$,  then by the definition of $\lambda_{0}(\Omega)$, for $h:=(\lambda_{0}(\Omega)/\mu(\Omega))^{1/2}$,
$$\int_{\Omega}|f|h\d \mu\leqslant\sqrt{\mu(f^2\mathbf{1}_\Omega)}\sqrt{\mu(h^2\mathbf{1}_\Omega)}\leqslant\sqrt{\lambda_{0}(\Omega)^{-1}\mathcal{E}(f,f)}\sqrt{\lambda_{0}(\Omega)}=\sqrt{\mathcal{E}(f,f)}
$$
for all $f\in \mathcal{F}^\Omega$, which implies the Dirichlet form $(\mathcal{E},\mathcal{F}^\Omega)$ is transient (see \cite[(1.5.6)]{MYM11}).
Therefore, from \cite[Theorem 1.3.9]{O13}, one could find   \eqref{poi} exists a unique weak solution
$$
G^\Omega\xi(x):=\int_\Omega \xi(y)G^\Omega(x,\d y)=\mathbb{E}_x\int_0^{\tau_\Omega} \xi(X_t)\d t,\quad x\in E,
$$
under the assumpation $\lan \xi, G^\Omega\xi\ran_2<\infty$,
where $G^\Omega(x,\d y)$ is the Green potential measure of process $X$ killed upon $\Omega$, that is,
$$
G^\Omega(x,\d y):=\int_0^\infty \mathbb{P}_x(X_t\in \d y, t<\tau_\Omega)\d t.
$$

Using the above analysis, the following lemma tells us that the $k$-th moments of the exit time satisfy some Poisson equations.

\begin{lemma}\label{exit-poisson}
 Let $\Omega\subset E$ be a domain with  $\mu(\Omega)<\infty$ and $\lambda_0(\Omega)>0$. Define $\{u_k\}_{k\geqslant 0}$ with $u_0=1$ and for $k\geqslant 1$,
$$u_k(x):=k\int_\Omega u_{k-1}(y)G^\Omega(x,\d y),\quad x\in E.$$
Then for each $k\geqslant 1$, $u_k(x)=\mathbb{E}_x\tau_\Omega^k,x\in E$. Moreover, it is the weak solution of Poisson equation
\begin{equation}\label{poisson}
\begin{cases}
	-\mathcal{L}u_k=ku_{k-1},&\q \text{in }\Omega;\\
	u_k=0,&\q \text{q.e. on }\Omega^c.
\end{cases}
\end{equation}

\end{lemma}

\begin{proof}

First, we prove $u_k=(\mathbb{E}_x\tau_\Omega^k:x\in E),\ k\geqslant 1$ by induction.
In fact, when  $k=1$,
$$
u_1(x)=\int_\Omega G^\Omega(x,\d y)=\int_{0}^{\infty}\mathbb{P}_x(\tau_\Omega>t)\d t=\mathbb{E}_x\tau_{\Omega}.
$$

Assume now that for $k=n\geqslant 1$, the assertion holds. Then for $k=n+1$,
\begin{equation}\label{u_{k+1}}
	\begin{split}
		u_{n+1}(x)&=(n+1)\int_\Omega u_{n}(y)G^\Omega(x,\d y)=(n+1)\int_\Omega \mathbb{E}_y\tau_{\Omega}^nG^\Omega(x,\d y)\\
		&=(n+1)\int_0^\infty\int_\Omega\mathbb{E}_y\tau_{\Omega}^n\mathbb{P}_x(X_t\in \d y,t<\tau_\Omega)\d t\\
		&=n(n+1)\int_0^\infty\int_\Omega\int_0^\infty r^{n-1}\mathbb{P}_y(\tau_{\Omega}>r)\d r\mathbb{P}_x(X_t\in \d y,t<\tau_\Omega)\d t.
	\end{split}
\end{equation}
Note that by the strong Markov property, for any $t,r>0$ we have
\begin{align*}
		\mathbb{P}_x(\tau_\Omega>t+r)&=\mathbb{E}_x[\mathbf{1}_{\{\tau_\Omega>t\}}\mathbb{P}_{X_t}(\tau_\Omega>r)]=\int_{\Omega}\mathbb{P}_y(\tau_{\Omega}>r)\mathbb{P}_x[X_t\in \d y,\tau_{\Omega}>t].
\end{align*}
Applying the above analysis to \eqref{u_{k+1}} gives that
\begin{equation*}
	\begin{split}
		u_{n+1}(x)&=n(n+1)\int_0^\infty\int_0^\infty r^{n-1}	\mathbb{P}_x(\tau_\Omega>t+r)\d r\d t\\
		&=n(n+1)\int_0^\infty\int_r^\infty r^{n-1}	\mathbb{P}_x[\tau_\Omega>s]\d r\d s\\
		&=(n+1)\int_0^\infty\int_0^s r^{n-1}	\mathbb{P}_x(\tau_\Omega>s)\d r\d s\\
		&=\int_0^\infty(n+1) s^n	\mathbb{P}_x(\tau_\Omega>s)\d s=\mathbb{E}_x\tau_{\Omega}^{n+1}.
	\end{split}
\end{equation*}

Since $\lambda_0(\Omega)>0$, it is clear that   the Dirichlet space $(\mathcal{E},\mathcal{F}^\Omega)$ is transient. Note that by \cite[Lemma 3.1]{HKMW21+}, for $\beta<\lambda_{0}(\Omega)$,  $\mathbb{E}_x\mathrm{e}^{\beta\tau_\Omega}\in L^2(\mu)$  for all $x\in E$. Therefore, $u_k=\mathbb{E}_x[\tau_\Omega^k]\in L^2(\mu)$, which implies that $\lan u_k, G^\Omega u_k\ran_2=(k+1)^{-1}\lan u_k,  u_{k+1}\ran_2<\infty.$ Thus  by \cite[Theorem 2.1.12]{CF12},  we see that $u_k\in\mathcal{F}^\Omega$ and
\begin{equation}\label{diedai}
\mathcal{E}(u_k,v)=k\lan u_{k-1},f\ran_2 \quad\text{for all }f\in\mathcal{F}^\Omega.
\end{equation}
Hence, $u_k$ is the weak solution of \eqref{poisson}.
\end{proof}

From Lemma \ref{exit-poisson} and the proof of \cite[Theorem 2.1]{HKMW21+}, we obtain the following variational formulas for the $k$-th moments of the exit time directly.

\begin{proposition}\lb{VF}
	Let $\Omega\subset E$ be a domain with $\mu(\Omega)<\infty$ and $\lambda_0(\Omega)>0$, and $\{u_k\}_{k\geqslant 0}$ be defined as in Lemma \ref{exit-poisson}. Then for any $k\geqslant 1$,
\begin{equation*}
	\frac{1}{k\lan u_{k-1},u_k\ran_2}= \inf_{f\in \cM_{\Omega,k}}\mathcal{E}(f,f),
\end{equation*}
where $
\cM_{\Omega,k}:=\{f\in\mathcal{F}^\Omega:k\lan u_{k-1},f\ran_2=1\}.
$
\end{proposition}

\section{Proofs}\label{proofs}
To prove our main results, we need the following lemma.
\begin{lemma}
Let $\Omega\subset E$ be a domain with $\mu(\Omega)<\infty$ and $\lambda_0(\Omega)>0$, and $\{u_k\}_{k\geqslant 0}$ be defined as in Lemma \ref{exit-poisson}. Denote $T_k(\Omega)=\int_\Omega u_k(x)\mu(\d x)$. Then for any $k\geqslant 1$,
\begin{equation}\label{iterate}
k\lan u_{k-1},u_k\ran_2=\frac{(k !)^{2}}{(2 k-1) !} T_{2k-1}(\Omega)
\quad \text{and}\quad
	\lan u_{k},u_k\ran_2=\frac{(k !)^{2}}{(2 k) !} T_{2k}(\Omega).
\end{equation}
\end{lemma}

\begin{proof}
For fixed $k\geqslant 1$, we apply \eqref{diedai} to $u_{k+1}$ and $u_{k-1}$, and obtain that
\begin{equation*}
	\begin{split}
		\lan u_k,u_{k-1}\ran_2&=\frac{1}{k+1}\mathcal{E}(u_{k+1},u_{k-1})
		=\frac{k-1}{k+1}\lan u_{k-2},u_{k+1}\ran_2.
	\end{split}
\end{equation*}
Iterating this procedure, we have
$$\lan u_k,u_{k-1}\ran_2=\frac{(k-1)!k !}{(2 k-1) !} \int_{\Omega} u_{2 k-1} \d \mu,$$
which implies the first equality in \eqref{iterate}.

Similarly, from \eqref{diedai},
\begin{equation*}\label{iterate2}
		\lan u_{k},u_k\ran_2=\frac{1}{k+1}\mathcal{E}(u_{k+1},u_{k})=\frac{k}{k+1}\lan u_{k-1},u_{k+1}\ran_2,\\
\end{equation*}
we can complete the proof by iterating.
\end{proof}

We are now ready to prove Theorem \ref{main}.

\noindent{\bf Proof of Theorem \ref{main}.}
	By \eqref{diedai} and \eqref{iterate}, we see that
$$
	\mathcal{E}(u_k,u_k)=k\lan u_{k-1},u_k\ran_2=\frac{(k !)^{2}}{(2 k-1) !} T_{2k-1}(\Omega).
$$
Therefore, taking $f=u_k$ in \eqref{Diri-eigen}, we get
\begin{equation}\label{upperb1}
\lambda_{0}(\Omega)\leqslant\frac{\mathcal{E}(u_k,u_k)}{\int_{\Omega} u_{k}^{2} \d \mu}
\leqslant \frac{(k !)^{2}}{(2 k-1) !}  \frac{T_{2 k-1}(\Omega)}{\int_{\Omega} u_{k}^{2} \d \mu}
\leqslant \frac{(k !)^{2}}{(2 k-1) !} \frac{T_{2 k-1}(\Omega)}{T_{k}^{2}(\Omega)}\mu(\Omega),
\end{equation}
where in the last inequality we used the H\"older inequality
$$
T_k^2(\Omega)=\left(\int_\Omega u_k\d \mu\right)^2\leq \mu(\Omega)\int_\Omega u_k^2\d \mu.
$$

On the other hand, if we apply the second equality in \eqref{iterate} to \eqref{upperb1}, we also have that
\begin{equation}\label{moment2}
\lambda_{0}(\Omega)\leqslant\frac{\frac{(k !)^{2}}{(2 k-1) !} T_{2 k-1}(\Omega)}{\frac{(k !)^{2}}{(2 k) !} T_{2 k}(\Omega)  }=2 k \frac{T_{2 k-1}(\Omega)}{T_{2 k}(\Omega)}.
\end{equation}

To complete the proof of \eqref{moment3}, we need the variational formulas in Proposition \ref{VF}. Specifically, for fixed $k\geq 1$ and $f\in \mathcal{M}_{\Omega,k}$,  by H\"older inequality,
\begin{equation}
 1=\lan k u_{k-1},f\ran_2^2\leqslant k^2\mu(u_{k-1}^2)\mu(f^2).
\end{equation}
Thus from \eqref{Diri-eigen} we have
$$
\lambda_{0}(\Omega)\leqslant\frac{\mathcal{E}(f,f)}{\mu(f^2)}\leqslant \mathcal{E}(f,f)k^2\mu(u_{k-1}^2).$$
According to the arbitrariness of $f\in \cM_{\Omega,k}$  and Proposition \ref{VF},
\begin{equation*}
	\lambda_{0}(\Omega)\leqslant\frac{k^2\mu( u_{k-1}^2)}{k\lan u_{k-1},u_k\ran_2}=\frac{k\mu( u_{k-1}^2)}{\lan u_{k-1},u_k\ran_2}.
\end{equation*}
Combining this with \eqref{iterate} yields
\begin{equation}\label{moment4}
\lambda_{0}(\Omega)\leqslant\frac{k\frac{((k-1) !)^{2}}{(2 k-2) !} T_{2k-2}(\Omega)}{\frac{(k !)^{2}}{k(2 k-1) !} T_{2k-1}(\Omega)}=(2k-1)\frac{T_{2k-2}(\Omega)}{T_{ 2k-1}(\Omega)}.
\end{equation}
Therefore, \eqref{moment3} is obtained by \eqref{moment2} and \eqref{moment4}. By iterating and the Taylor expansion of function $\text{e}^{\beta x}$, we could obtain \eqref{k-moment} immediately. \quad \qed

\noindent{\bf Proof of Theorem \ref{main-3}.}
By the assumption that $-\mathcal{L}^\Omega$ possesses the discrete spectrum $0<\lambda_0(\Omega)\leqslant\lambda_1(\Omega)\leqslant \cdots$, and the fact $P_t^{\Omega}=\text{e}^{t\mathcal{L}^\Omega}$, we see that $P_t^{\Omega}$ admits a spectral decomposition:
$$
	P_t^{\Omega} f=\sum_{i=0}^\infty \text{e}^{-\lambda_i(\Omega)t}\mu(\varphi_{i}f)\varphi_{i}\quad
	\text{for }f\in L^2(\mu)\  \text{with } f=0\ \mu\text{-a.e. on }\Omega.
$$
Together this with the proof of  Lemma \ref{exit-poisson}, for any $k\geqslant 1$ we get that
	\begin{align*}
		\mathbb{E}_x[\tau_\Omega^k]&=k \int_0^\infty s^{k-1}\mathbb{P}_x(\tau_\Omega>s)\d s
		=k \int_0^\infty s^{k-1} P_s^{\Omega}{\bf 1}(x)\d s\\
		&=k\sum_{i=0}^\infty \int_{\Omega}\varphi_i\d \mu
		\int_0^\infty s^{k-1}\text{e}^{-\lambda_i(\Omega)s}\d s\varphi_i\\
		&=k!\sum_{i=0}^\infty \frac{\int_{\Omega}\varphi_i\d \mu}{(\lambda_i(\Omega))^k}\varphi_i,
	\end{align*}
	which implies
$$
	T_k(\Omega)=k!\sum_{i=0}^\infty \frac{\left(\int_{\Omega}\varphi_i\d \mu\right)^2}{(\lambda_i(\Omega))^k}
\geqslant \frac{k!\left(\int_{\Omega}\varphi_0\d \mu\right)^2}{(\lambda_0(\Omega))^k}.
$$
\qed

\noindent{\bf Proof of Corollary \ref{limit}.}
Denote by
$$\eta=\sup \left\{\lambda\geqslant 0: \limsup _{n \rightarrow \infty}\lambda^n \frac{T_n(\Omega)}{n!}<\infty\right\}.$$
First, by Theorem \ref{main},
$0<\mu(\varphi)^2\leqslant T_{ k}(\Omega)\lambda_0(\Omega)^k/k!\leqslant\mu(\Omega)<\infty$, which implies that 	$\lambda_0(\Omega)\leqslant\eta$.
On the other hand, for any $\lambda>\lambda_{0}(\Omega)$, 
$$\lambda^k\frac{T_{ k}(\Omega)}{k!}=\left(\frac{\lambda}{\lambda_{0}(\Omega)}\right)^k\left(\lambda_0(\Omega)^k\frac{T_{ k}(\Omega)}{k!}\right)\geqslant \left(\frac{\lambda}{\lambda_{0}(\Omega)}\right)^k\mu(\varphi)^2\rightarrow\infty\quad \text{as}\ k\rightarrow\infty,$$
therefore, $\lambda_0(\Omega)\geqslant\eta$, thus $\lambda_0(\Omega)=\eta$. 
The remaining part can be obtained by a similar argument. \qed

\section{Examples}\label{examples}

In this section, we will present some concrete examples and provide some estimates of their Dirichlet eigenvalues and the exit time moments from our main results.

\subsection{Diffusion processes on Riemannian manifolds}\label{diffusion}
Let $M$ be a $d$-dimensional connected complete Riemannian manifold  with Riemannian metric $\rho$ and Riemannian volume $\d x$.
Consider the diffusion operator $L=\Delta+ \nabla V\cdot\nabla$
where $V\in C^2(M)$, and let $Y$ be the associated ergodic diffusion process on $M$ with the generator $L$ and stationary distribution $\pi(\d x) =\exp{(V (x))}\d x/\int_M \exp{(V (x))}\d x$.
For fixed $o\in M$, let $\rho(x)$ be the Riemannian distance function from $o$, and   $\mathrm{cut}(o)$ be the cut-locus. Assume that either $\partial M$ is bounded or $M$ is convex.  Let $D$ is the diameter of $M$.
In this case, the associated Dirichlet form is given by
$$
\mathcal{E}(f,g)=\int_M\nabla f\cdot\nabla g\d\pi\quad \text{and}\quad \mathcal{F}=C^1(M)\cap L^2(M,\pi).
$$
Set $B_r=\{x\in M:\rho(x)\leqslant r\}$,
	$$\gamma(r)=\sup_{\rho(x)=r,x\notin \mathrm{cut}(o)}L\rho(x)\q\text{and}\q C(r)=\int_{1}^{r}\gamma(s)\d s.$$
	Assume that
	$$
		\delta_r:=\sup_{t\geq r}\int_{r}^{t}\exp{(-C(l))\d l}\int^{D}_{ t}\exp{(C(s))}\d s<\infty.
	$$
Then \cite[(3.13)]{HKMW21+} tells us that
	$$
		\lambda_0(B_r)\geqslant 1/(4\delta_r)>0.
	$$
	Thus plugging this into \eqref{k-moment} yields that
	$$T_k(B_r)\leqslant\frac{k!\mu(B_r)}{(\lambda_0(\Omega))^k}\leqslant 4k!\mu(B_r)\delta_r,$$
and $$
\mathbb{E}_\mu[\mathrm{e}^{\beta\tau_{B_r}}]\leqslant\left(1+\frac{4\beta\delta_r}{1-4\beta\delta_r}\right)\mu(B_r)\q \text{for }0<\beta< 1/(4\delta_r).
$$	

\subsection{Symmetric stable processes}
Let $Z:=(Z_t)_{t\geqslant 0}$ be a symmetric $\alpha$-stable process on $\mathbb{R}^d$ with   generator $-(-\Delta)^{\alpha / 2}$, $\alpha\in (0,2]$, where $-(-\Delta)^{\alpha / 2}$ is the fractional Laplacian.
Note that when $\alpha=2$ the process $Z$ is a $d$-dimensional Brownian motion running at twice the usual speed.
It is known that the process $Z$ is symmetric with respect to Lebesgue measure on $\mathbb{R}^d$. Fix a bounded open subset $\Omega\subset\mathbb{R}^d$, and denote by $X^\Omega$ the sub-process killed upon leaving $\Omega$ with transition kernel
$$
P_{t}^{\alpha,\Omega} (x,A):=\mathbb{P}_{x}\left(Z_{t}\in A, t<\tau_\Omega^{(\alpha)}\right),\quad x\in \Omega, A\in \mathscr{B}(\mathbb{R}^d),$$
where   $\tau_\Omega^{(\alpha)}$ is the exit time of process $Z$ from $\Omega$.

Let $\mathcal{L}_\alpha^\Omega$ be the associated generator of $P_t^{\alpha,\Omega}$.
Since $\Omega$ is bounded, the spectrum of $-\mathcal{L}_\alpha^\Omega$ is discrete (see \cite{DM07} for more details), denoted by
$$0<\lambda^{(\alpha)}_{0}(\Omega)\leqslant\lambda_{1}^{(\alpha)}(\Omega)\leqslant\lambda_{2}^{(\alpha)}(\Omega)\leqslant\cdots,$$

Let $\varphi_{0}^{(\alpha)}, \varphi_{1}^{(\alpha)}, \varphi_{2}^{(\alpha)},  \ldots$
be the corresponding orthonormal $L^{2}(\Omega)$ eigenfunctions with $\text{Leb}((\varphi_{i}^{(\alpha)})^2):=\int_\Omega\varphi_{i}^{(\alpha)}(x)^2\d x=1$ for all $i\geqslant 0$.
Denote by $T_k^{(\alpha)}(\Omega)=\int_\Omega \mathbb{E}_x[(\tau_\Omega^{(\alpha)})^k]\d x$ the $k$-th moment of $\tau_\Omega^{(\alpha)}$ .

From \cite[Theorems 4 and 5]{BLM01} we know that  if $\Omega\subset \mathbb{R}^d$ is a bounded domain, then the first Dirichlet eigenvalue $\lambda_0^{(\alpha)}(\Omega)$ has estimate
$$
 \frac{\gamma_{d}^{\alpha / d} 2^{\alpha} \Gamma\left(1+\frac{\alpha}{2}\right) \Gamma\left(\frac{d+\alpha}{2}\right)}{\Gamma\left(\frac{d}{2}\right) \operatorname{Leb}(\Omega)^{\alpha / d}}\leqslant \lambda_0^{(\alpha)}(\Omega)\leqslant \left(\lambda_0^{(2)}(\Omega)\right)^{\alpha/2},
$$
where $\gamma_d$ is the volume of the unit ball in $\mathbb{R}^d$.
 Therefore, together this fact with Theorems \ref{main} and \ref{main-3}, we arrive at
$$
\frac{k!\left(\text{Leb}(\varphi_0^{(\alpha)})\right)^2}{(\lambda_0^{(2)}(\Omega))^{\alpha k/2}}
\leqslant
T^{(\alpha)}_{ k}(\Omega)\leqslant  \frac{k!\left(\text{Leb}(\Omega)\right)^{(k\alpha / d)+1}\Gamma\left(\frac{d}{2}\right)^k }{\left(\gamma_{d}^{\alpha / d} 2^{\alpha} \Gamma\left(1+\frac{\alpha}{2}\right) \Gamma\left(\frac{d+\alpha}{2}\right)\right)^k}.$$

\subsection{Time-changed symmetric stable processes}
Let $Z:=(Z_t)_{t\geqslant 0}$ be the symmetric $\alpha$-stable process on $\mathbb{R}$ with   generator $-(-\Delta)^{\alpha / 2},\ \alpha\in(0,2)$.
Consider the following stochastic differential equation:
\begin{equation}\label{SDE}
	\mathrm{d} X_{t}=\sigma\left(X_{t-}\right) \mathrm{d} Z_{t},
\end{equation}
where $\sigma$ is a strictly positive continuous function on $\mathbb{R}$. By \cite[Proposition 2.1]{DK20}, there is a unique weak solution $X=(X_t)_{t\geqslant 0}$ to the SDE \eqref{SDE}, and $X$ can also be expressed as a time change process $
X_{t}:= Z_{\zeta_{t}},$ {where}  $$\zeta_{t}:=\inf \left\{s > 0: \int_{0}^{s} \sigma\left(Z_{u}\right)^{-\alpha} \mathrm{~d} u>t\right\}.
$$
Furthermore, from \cite[Section 1.2]{CW14} we see that the generator of  process $X$ is  $L=\sigma^\alpha\Delta^{\alpha / 2}$,  which is symmetric with respect to its invariant measure $\mu(\mathrm{d} x)=\sigma(x)^{-\alpha} \mathrm{d} x/\int_{\mathbb{R}}\sigma(x)^{-\alpha} \mathrm{d} x$.

As we know, \cite[Theorem 2]{W21+} shows that
$$\lambda_{0}((0, \infty)) \geqslant \frac{(\alpha-1)\Gamma(\alpha / 2)^{2}}{4\delta_+},$$
where $$\delta_{+}:=\sup _{x>0} x^{\alpha-1} \int_{x}^{\infty} \sigma(z)^{-\alpha} \mathrm{d} z<\infty.$$
Thus combining it with Theorem \ref{main}, we have $$T_k((0,\infty))\leqslant\frac{k!\mu((0,\infty))}{(\lambda_0((0,\infty)))^k}\leqslant\frac{4^k k!\mu((0,\infty))}{(\alpha-1)^k\Gamma(\alpha / 2)^{2k}}\delta_+^k,$$
and
$$
\mathbb{E}_\mu[\mathrm{e}^{\beta\tau_{(0,\infty)}}]\leqslant\left(1+\frac{4\beta\delta_+}{(\alpha-1)\Gamma(\alpha/2)^2-4\beta\delta_+}\right)\mu((0,\infty))\q \text{for }0<\beta< \frac{(\alpha-1)\Gamma(\alpha / 2)^{2}}{4\delta_+}.
$$

{\bf Acknowledgement}\
Lu-Jing Huang acknowledges support from NSFC (No. 11901096) and NSF-Fujian(No. 2020J05036).
Tao Wang acknowledges the National
Key Research and Development Program of China (2020YFA0712900) and the project from the Ministry of Education in China.


\bibliographystyle{plain}
\bibliography{exit}

\end{document}